\documentclass[a4paper,10pt, oneside]{article}
\usepackage{amsmath, amssymb, amsthm}
\usepackage[titletoc, title]{appendix}
\usepackage{mathrsfs}
\usepackage{paralist}
\usepackage{graphics} 
\usepackage{epsfig} 
\usepackage{multirow}
\usepackage{graphicx}
\usepackage{epstopdf}
\graphicspath{{DATA/}}
\usepackage[pdftex, pdfpagelabels, bookmarks, hyperindex, hyperfigures,
colorlinks, pagebackref]{hyperref}
\usepackage[capitalise, nosort]{cleveref}

\crefname{equation}{}{}
\crefname{lem}{Lemma}{Lemmas}
\crefname{thm}{Theorem}{Theorems}

\newcommand{\snmii}[1]
{
  \left\vert\kern-0.25ex
  \left\vert\kern-0.25ex
  \left\vert
  #1
  \right\vert\kern-0.25ex
  \right\vert\kern-0.25ex
  \right\vert
}

\newtheorem{lem}{Lemma}[section]
\newtheorem{rem}{Remark}[section]
\newtheorem{thm}{Theorem}[section]
\newtheorem{example}{Example}[section]

\numberwithin{equation}{section}

\title{Extended finite element methods for optimal control problems governed by Poisson equation in non-convex domains
	\thanks
	{
		This work was supported by  National Natural Science Foundation of China (11771312).
}}
\author{
	\ Tao Wang  \thanks{School of Mathematics, Sichuan University, Chengdu 610064, China. Email: wangtao5233@hotmail.com }, Chao Chao Yang \thanks{School of Science, Chongqing University of Technology, Chongqing 400054, China. Email: yangchaochao9055@163.com}, \
	Xiaoping Xie \thanks{Corresponding author. School of Mathematics, Sichuan University, Chengdu 610064, China. Email: xpxie@scu.edu.cn}
}
\date{}
\begin{document}
		\maketitle
		
			\begin{abstract}
This paper  analyzes two eXtended finite element methods (XFEMs) for  linear quadratic  optimal control problems governed by Poisson equation in non-convex domains.   We follow the variational discretization concept  to discretize the continuous problems, and apply  an XFEM with a cut-off function and a classic XFEM with a fixed enrichment area to   discretize the state and co-state equations.  Optimal error estimates are derived for the state, co-state and control. Numerical results confirm our theoretical results.
		\end{abstract}	
	

		\section{Introduction}
We consider the following linear quadratic optimal control problem: 
\begin{equation}\label{eqobjective}
       \text{min}~ J(y, u):=\frac{1}{2}\int_\Omega (y-y_d)^2~dx+\frac{\alpha}{2}\int_\Omega u^2~dx
\end{equation}
for $ (y,u)\in H^1_0(\Omega)\times L^2 (\Omega) $ subject to Poisson equation
\begin{equation}\label{eqstrongstate}
\left\{
 \begin{array}{rll}
 & -\Delta y=u+f & \text{ in }\Omega, \\
 & y=0 & \text{ on }\partial\Omega, \\
\end{array}
\right.
\end{equation}
with the control constraint
\begin{equation}\label{eqcontrol1}
u_0\leq u \leq u_1, \text{ a.e. on } \Omega,
\end{equation}
	where $\Omega$ is a bounded polygonal domain in $\mathbb{R}^2$ with a single re-entrant corner of angle $\frac{\pi}{\beta}$, $\beta\in[\frac12,1]$. $y_d\in L^2(\Omega)$ is the desired state to be achieved by controlling  $u$. $\alpha$ is a positive constant and  $f, u_0, u_1 \in L^2 (\Omega)$ with $u_0\leq u_1$ a.e. on $\Omega$. For the sake of simplicity, we choose homogeneous
	boundary condition on $\partial\Omega$. In fact, we can obtain similar results for other boundary conditions.
	
	For the Poisson problem \eqref{eqstrongstate}  in non-convex domains,  it is well-known that the weak solution $y$ is generally not in $H^2(\Omega)$, due to the singularities at the corner. The low regularity  may lead to  reduced accuracy for finite element approximations \cite{Babu1972A}.   In literature there are two ways to improve the accuracy. The first way  is to use  graded meshes  (cf. \cite{Oganesyan1968Variational,Babuka1979Direct,Schatz1979Maximum,Apel1996Graded}).  The second  way is to use some singular basis functions which characterize the singularity of the solution around  the corner; see, for instance, the classic singular enrichment method $\cite{Strang1973An}$, the dual singular function method \cite{Blum1982On}, and the singular complement method  \cite{Jr2003The}.  
	 Notice that when $\Omega$ is a  crack domain ($\beta=\frac12$),  
	 all the above methods have to use  body-fitted meshes. However, it is often difficult or expensive to construct such kinds of meshes, especially in  time dependent problems. 
	 
%
%
	 
	 In the past few decades, many numerical methods have been developed for the optimal control problem \eqref{eqobjective}-\eqref{eqcontrol1} in convex domains, see \cite{Casas86con,Falk73app,Tiba76err,Becker00ada,Li02ada,Schneider16pos,Benedix09pos,mmuller10goa,Rosch12pri,Rosch12pos,Rosch17rel,Gong17ada,Li02ada, Hinze05var,Hinze2009Variational,Meyer2003Superconvergence,Yang2018}.  
However, 	 for   optimal control 
problems in non-convex domains, there is only limited  research work. In \cite{Apel2007Optimal,Apel2012Finite,Apel2009Optimal},  finite element error estimates were derived on  graded meshes.
	 
The extended finite element method (XFEM, also called generalized finite element method(GFEM)) is known to be a widely-used tool 
  for the analysis of problems with 
singularities \cite{Belytschko1999Elastic, moes1999a, Strouboulis00des, Strouboulis2006The,Belyschkos09rev,Soghrati2012A,babuska1997the,Kergrene2016Stable,Babuska2010Optimal,Berrone2013On,Babu2017Strongly}. With additional basis functions characterizing the singularity  added into the standard approximation space,   XFEM  does not need body-fitted meshes, and thus avoid complicated meshes. 
	
	 In this paper, we consider an XFEM for the optimal control \eqref{eqobjective}-\eqref{eqcontrol1} problem in non-convex domains. We follow the variational discretization concept $\cite{Hinze05var,Hinze2009Variational}$ to discretize the continuous problem. 
	 Optimal error estimations are derived for the state, co-state and control.
	 We apply the semi-smooth Newton method  to solve the resultant nonlinear discrete system. 
	 
	The rest of the paper is arranged as follows. Section 2  gives some notations, the optimality conditions, and regularity results for the optimal control problem. Section 3  introduces the XFEM, and shows several theoretical results associate with XFEM.  Section 4 is devoted to the discrete optimal control problem,  the discrete optimality conditions and   error estimates  for the state, co-state and control.  Section 5 gives an iteration algorithm for the discrete system. Finally, Section 6 provides numerical results to verify the theoretical analysis. 
	\section{Preliminary}
	 Let  $\Omega$ be a polygonal domain   with   a single re-entrant corner of angle $\frac{\pi}{\beta}$, $\beta\in[\frac12,1]$. And (0,0) is the concave point of $\Omega$.
For  
any nonnegative integer $m$, let $H^{m}(\Omega)$ and $H^{m}_0(\Omega)$  denote the standard    Sobolev spaces on $\Omega$ with    norm $\|\cdot\|_{m}$ and semi-norm $|\cdot|_{m}$. In particular, $H^0(\Omega)=L^2(\Omega)$, with  the standard $L^2$-inner product $(\cdot,\cdot)$. 
	 
	 For $u\in L^2(\Omega)$, the weak formulation of the state equation \eqref{eqstrongstate} is  as follows:  find $y \in H^1_0(\Omega)$ such that
	 \begin{equation}
	 \label{eqweakform}
	a(y,v) =(u+f,v), \quad \forall v \in H^1_0(\Omega),
	 \end{equation} 
	where  $a(y,v):=( \nabla y ,\nabla v)$. 

 \begin{figure} [!hbtp]
	\centering
	\includegraphics[width=8cm]{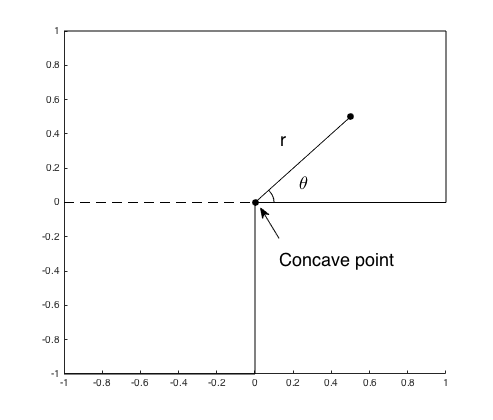}
	\caption{Polar coordinates.}
	\label{fgpolarcor}
\end{figure}
	 
	Introduce a singular function 
 		\begin{equation*}
 		S_\beta(r,\theta)= r^{\beta}sin{(\beta\theta)} .
 		\end{equation*} 
			It is well known that the weak solution $y$ is generally not in $H^2(\Omega)$. According to   $\cite{Blum1982On,Grisvard1985Elliptic}$, $y$ has a singular part  
	 \begin{equation}
	 \label{ys}
	 y_s=\chi(r) K_yS_\beta(r,\theta)
	 \end{equation} 
	 such that 
		  \begin{equation}
	 \label{y-ys}
	 y-y_s  \in H^2(\Omega).
 \end{equation} 
		 Here  $r$ and $\theta$ are polar coordinates 
		 with $0\leq\theta\leq 2\pi$  (cf. Figure \ref{fgpolarcor} for an $L$-shaped domain).  $K_y$ is a positive constant depending on $u$ 
		and  
		 can be regarded as a linear functional of $u+f$ (cf. \cite{Jr2003The,Shen2010Stability}). $\chi\in W^{2,\infty}(0,\infty)$ is a cut-off function defined by
	 	\begin{equation}\label{chi}
		\left\{\begin{array}{ll}
	 	\chi(r)=1, &  \quad if \  r<r_0, \\
	 	0<\chi(r) < 1, & \quad if \ r_0 < r < r_1, \\
	 	\chi(r)=0,& \quad if \ r_1 < r,
	 	\end{array}
		\right.
		\end{equation}
	where   $r_0,r_1$ are two   constants with  $0<r_0<r_1$. 	And, throughout  the paper, we use ``$p\lesssim q$'' to denote ``$p \leq C q$, where  $C$ is a generic positive constant   independent of the solution $u$ and the finite element mesh size  $h$. 

		In view of \eqref{ys}-\eqref{y-ys}, we assume the following regularity on  the solution $y$  to the problem \eqref{eqweakform} (cf. \cite{Shen2010Stability}):
\begin{equation}		 
	 		\label{leregularity}
	 		K_y+\|y- y_s\|_2\lesssim \|u\|_0+ \|f\|_0.
	\end{equation}		
	
	\begin{rem}\label{rem2.1}
	 Set 
	\begin{equation}
	\label{ys-1}
	\widetilde{y_s}:=K_y S_\beta(r,\theta),
	\end{equation}
	 and let $B(r)$ be  a ball with center $(0,0)$ and radius $r$. Then, for any given enrichment radius  $r_s>0$, 
	 by \eqref{leregularity} we easily get 
 		\begin{align}\label{leregularity1}
 			\|y-\widetilde{y_s}\|_2+\|\widetilde{y_s}\|_1 + \|\widetilde{y_s}\|_{2,\Omega \setminus B(r_s/2)} & \lesssim \|u\|_0+ \|f\|_0 .
 		\end{align} 
	In fact, from \eqref{leregularity} it follows 
	 $$\|\widetilde{y_s}\|_1 + \|\widetilde{y_s}\|_{2,\Omega \setminus B(r_s/2)}=K_y(\|S_\beta\|_1+\|S_\beta\|_{2,\Omega \setminus B(r_s/2)})  \lesssim \|u\|_0+ \|f\|_0.$$ 
	By the definition of the cut-off function $\chi(r)$ we have 
	$$ \|(1-\chi)\widetilde{y_s}\|_2= \|(1-\chi)\widetilde{y_s}\|_{2,\Omega \setminus B(r_0)}\leq \|\widetilde{y_s}\|_{2,\Omega \setminus B(r_0)} \lesssim \|u\|_0+ \|f\|_0,$$
	which, together with   the triangle inequality, yields
	 $$\|y-\widetilde{y_s}\|_2 \leq \|y-y_s\|_2+\|(1-\chi)\widetilde{y_s}\|_2 \lesssim \|u\|_0+ \|f\|_0.$$
	Hence, the estimate \eqref{leregularity1} holds.

	\end{rem}
%
 		
		By Following the standard optimality technique in $\cite{Tr2010Optimal}$, we can easily get the optimality conditions of the optimal control problem \eqref{eqobjective}-\eqref{eqcontrol1}.
		
 			\begin{lem} \label{leoptimalcondition}
 			The optimal control problem   \eqref{eqobjective}-\eqref{eqcontrol1} has a unique solution $(y,p,u)\in H_0^1(\Omega)\times H_0^1(\Omega)\times U_{ad}$ such that 
 			\begin{align}
 			& a(y,v)=(u+f,v),~~\forall v\in H_0^1(\Omega),\label{eqstate}\\
 			& a(v,p)=(y-y_d,v),~~\forall v\in H_0^1(\Omega),\label{eqcostate}\\
 			& (p+a u,v-u) \geq 0,~~\forall v\in U_{ad},\label{eqprojection}
 			\end{align}
where \begin{align}\label{Uad}
U_{ad}:=\{ v\in L^2(\Omega):u_{a}\leq v \leq u_{b} \ \text{a.e. in} \ \Omega \}.
\end{align}
		\end{lem}
We note that  $p$ is called the co-state or adjoint state, and  \eqref{eqcostate} is   the co-state equation.  

\begin{rem}
Let  
\begin{equation}
 p_s:=\chi(r) K_pS_\beta(r,\theta),
 \end{equation}
 with $K_p>0$, be the  singular part of the solution $p$ to the  co-state equation  \eqref{eqcostate} such that 
 \begin{equation}		 
	 		\label{regularity-p0}
	 		K_p+\|p- p_s\|_2\lesssim \|y-y_d\|_0, 
	\end{equation}	
	which, together with the fact $\|y\|_1\lesssim \|u\|_0$, means
	 \begin{equation}		 
	 		\label{regularity-p}
	 		K_p+\|p- p_s\|_2 \lesssim   \|u\|_0+\|y_d\|_0+\|f\|_0.
	\end{equation}	
	Similar to Remark \ref{rem2.1}, set 
	$$\quad \widetilde{p_s}:=  K_pS_\beta(r,\theta),$$
	then from \eqref{regularity-p} we have
			\begin{align}\label{leregularity1-tilde p}
 			\|p-\widetilde{p_s}\|_2+\|\widetilde{p_s}\|_1 + \|\widetilde{p_s}\|_{2,\Omega \setminus B(r_s/2)} & \lesssim \|u\|_0+\|f\|_0+\|y_d\|_0.
 		\end{align} 

	\end{rem}
\begin{rem} 
 		The variational inequality \eqref{eqprojection} means that 
\begin{equation}\label{u-P}
u=P_{U_{ad}} \left (-\frac{1}{\alpha} p \right),
\end{equation} 
where $P_{U_{ad}}$ denotes the $L^2-$projection onto $U_{ad}$. \end{rem}
  		
 		\section{XFEM for state and co-state equations}
 	From Lemma \ref{leoptimalcondition},     the  state $y$ and the co-state $p$ can respectively be viewed as the solutions to the following two problems.

Find $y\in H_0^1(\Omega)$ such that
\begin{equation} \label{stat3}
a(y,v)=(u+f,v), ~~\forall v\in H_0^1(\Omega).
\end{equation}

Find $p\in H_0^1(\Omega)$ such that
\begin{equation} \label{co-stat3}
 a(v,p)=(y-y_d,v),~~\forall v\in H_0^1(\Omega).
\end{equation}

%
%
 		\subsection{Formulations of XFEM}
 		Let $\mathscr{T}_h$ be a shape-regular triangulation of $\Omega$ consisting of open triangles  with mesh size $h=\max_{K\in \mathscr{T}_h}h_K$, where $h_K$ denotes the diameter of $K\in \mathscr{T}_h$. Denote by $\Theta=\{a_i: i\in 1,2,\cdots, I\}$ the set   of   all the  vertexes  of all  triangles in $\mathscr{T}_h$.


		For $\forall a_i \in \Theta$, let $\varphi_i$ be the corresponding nodal basis function of the continuous linear finite element method with respect to   $\mathcal{T}_h$. 
Let $r_s>0$ be a prescribed constant called enrichment radius. We define 
 a vertex set 
$$\Theta_S:=\left\{a_i\in \Theta: \ \text{the distance between $a_i$   and the concave point is less than or equal to $r_s$}\right\}.$$
In particular, when   $\Omega$ is a cracked domain, i.e.  $\beta=\frac12$,  
we	set	
	$$\Theta_H :=
	\left\{a_i\in \Theta: \ \text{the support of $\varphi_i$ is completely cut by the crack of $\Omega$}\right\},
	$$
and define 
		the Heaviside function $H(x)$: for any $x=(r\cos\theta,r\sin\theta)\in \Omega$, 
 		\begin{equation*}
 		H(x)=
 		\begin{cases}
 		+1 &\mbox{if $x \cdot n \geq 0 $,} \\
 		-1 &\mbox{if $x\cdot n < 0$,}
 		\end{cases}
 		\end{equation*}
		where   $n$ is a unit normal vector along the crack. 
		
	With the above notations, set
	$$W_1:= span\{ \varphi_i: a_i\in \Theta\}+span\{ \chi (r) S_{\beta}\},$$
	$$W^*_1:= span\{ \varphi_i: a_i\in \Theta \} +span\{  \varphi_i H: a_i\in \Theta_H\}+span\{ \chi (r)S_{\frac12}\},$$
	$$W_2:= span\{ \varphi_i: a_i\in \Theta\}+span\{ \varphi _i S_{\beta}:a_i\in \Theta_S \},$$
	$$W^*_2:= span\{ \varphi_i: a_i\in \Theta \} +span\{  \varphi_i H: a_i\in \Theta_H\}+span\{ \varphi _i S_{\beta}:a_i\in \Theta_S \},$$
	and $V_0:=\{v: \ v=0 \ \text{on} \ \partial \Omega\}$. Then we 
 introduce  the following two extended finite element spaces:
	 $$V_h^1:=\left\{
	 \begin{array}{ll}\small 
	W_1\bigcap V_0  & \text{if}\  \frac12<\beta\leq1,\\\\
	W^*_1\bigcap V_0 & \text{if}\  \beta=\frac12,
	 \end{array}
	 \right. $$ 
	 $$V_h^2:=\left\{
	 \begin{array}{ll}\small 
	W_2\bigcap V_0  & \text{if}\  \frac12<\beta\leq1,\\\\
W^*_2\bigcap V_0	& \text{if}\  \beta=\frac12.
	 \end{array}
	 \right. $$  
 		
 		
 		It is easy to observe that 
		\begin{align}\label{subset}
		V_h^i\subset H_0^1(\Omega), \quad i=1,2.
		\end{align}
 		
		Take $V_h=V_h^1$ or $V_h^2$, then   
 		the XFEM formulations for the weak problems of the state $y$ and   co-state $p$ read as follows: 
		
 		Find $y^h\in V_h$ such that
		\begin{equation}
 		a(y^h, v_h)=(u+f,v_h), \quad \forall  v_h \in V_h. \label{eqdstate}
 		\end{equation}
		
		Find $p^h\in V_h$ such that
 		\begin{equation}
 		a(v_h, p^h)=(y-y_d,v_h), \quad \forall  v_h \in V_h.\label{eqdcostate}
 		\end{equation}
		
		\begin{rem}\label{rem3.1}
		The XFEM with  $V_h=V_h^1$  was called an XFEM with a cut-off function   \cite{Nicaise2011Optimal}, and   the one with $V_h=V_h^2$ was  called a  classic XFEM with a fixed enrichment area \cite{Belytschko1999Elastic}.

		\end{rem}
 	\subsection{Error estimates of XFEM}
	
	According to  \cite{Nicaise2011Optimal}, it holds the following error estimates.
 		
	\begin{lem}\label{lem3.1} Let $y,p$ be the solutions to the continuous problems \eqref{stat3} and  \eqref{co-stat3} respectively such that  the regularity conditions \eqref {leregularity} and \eqref{regularity-p} hold, and $y^h,p^h$ be the solutions to the discrete schemes \eqref{eqdstate} and \eqref{eqdcostate} respectively. Then
	the    estimates
			\begin{align}
 				\| y-y^h\|_1 & \lesssim h\|y-y_s\|_2\lesssim h (\|u\|_0+\|f\|_0),  \label{h1erry} \\ 
				\| p-p^h\|_1 & \lesssim h\|p-p_s\|_2\lesssim h\left( \|u\|_0+\|f\|_0+\|y_d\|_0\right)  \label{h1errp}
		\end{align}
		hold for $V_h=V_h^1,$ and     the estimates 
		\begin{align}
 				\| y-y^h\|_1 & \lesssim h\left(\|y-\widetilde{y_s}\|_2+\|\widetilde{y_s}\|_1 + \|\widetilde{y_s}\|_{2,\Omega \setminus B(r_s/2)} \right)\lesssim h (\|u\|_0+\|f\|_0), \\
				\| p-p^h\|_1 & \lesssim h\left(\|p-\widetilde{p_s}\|_2+\|\widetilde{p_s}\|_1 + \|\widetilde{p_s}\|_{2,\Omega \setminus B(r_s/2)} \right)\lesssim h\left( \|u\|_0+\|f\|_0+\|y_d\|_0\right)
 		\end{align}
		hold for $V_h=V_h^2$.
 		\end{lem}
		
Based on this lemma, we can  follow standard duality arguments to   derive      $L^2-$ estimates of   the errors $y-y^h$ and $p-p^h$.

 		\begin{lem} \label{lem3.2} Under the same conditions of Lemma \ref{lem3.1},   	the following estimates hold: 
			\begin{align}
 				\| y-y^h\|_0 & \lesssim h^2(\|u\|_0+\|f\|_0),
				\label{l2erry} \\ 
				\| p-p^h\|_0 & \lesssim h^2\left( \|u\|_0+\|f\|_0+\|y_d\|_0\right).  \label{l2errp}
		\end{align}
		\end{lem}
 	\begin{proof} We only show \eqref{l2erry} for $V_h=V_h^1,$ since the other cases follow similarly.  
	
	Consider the auxiliary problem
 		\begin{equation}
 		\left\{
 		\begin{array}{rlll}
 		 -\Delta z&=&y-y^h & \text{ in }\Omega, \\
 		 z&=&0 & \text{ on }\partial\Omega, \\
 		\end{array}
 		\right.
 		\label{eqauxiliary}
 		\end{equation}
 		which indicates
 		$$a(z,v)=(y-y^h,v), \quad \forall \ v \in H_0^1(\Omega).$$
 		Let $z^h\in V^1_h$ satisfy 
				$$a(z^h,v_h)=(y-y^h,v_h), \quad \forall \ v_h \in V_h^1,$$
and $z_s:= \chi(r) K_zS_\beta(r,\theta)$ be the regular part of $z$ with 
 \begin{align*}	 	
 	K_z+\|z- z_s\|_2\lesssim \|y-y^h\|_0. 
\end{align*}
 Then, similar to \eqref{h1erry}, it holds
 $$\| z-z^h\|_1  \lesssim h\|z-z_s\|_2\lesssim h \|y-y^h\|_0.$$
   As a result, 
 by the Galerkin orthogonality $a(z^h,y-y^h)=0$ and \eqref{h1erry}  we have 
 		\begin{align*}
 		\|y-y^h\|_0^2 &=a(z-z^h,y-y^h)\\
		&\leq \|z-z^h\|_1 \|y-y^h\|_1 \\
 		& \lesssim h^2 \|y-y^h\|_0( \|u\|_0+\|f\|_0),
 		\end{align*}	
 which yields \eqref{l2erry}.	
 	\end{proof}

 	
 	\section{Discrete optimal control problem}
 	\subsection{Discrete optimality conditions}
 	In this subsection, we follow the variational discretization concept $\cite{Hinze05var}$ to discretize the optimal control problem \eqref{eqobjective}-\eqref{eqcontrol1}. 
	The corresponding discrete optimal control problem is of the form  
 	\begin{equation} \label{eqdobjective}
 	\min\limits_{(y_h,u) \in V_h \times U_{ad}}J_h(y_h,u)=\frac{1}{2} \int_{\Omega} (y_h-y_d)^2 ds+\frac{\alpha}{2}\int_{\Omega} u^2 ds,
 	\end{equation} 
 	where $y_h = y_h(u)$ satisfies
 	\begin{equation}\label{eqdstate2}
 	a(y_h,v_h)=(u+f,v_h), \quad \forall v_h \in V_h.
 	\end{equation}
 Similar to the continuous case,  we have  the following existence and uniqueness result and optimality conditions.
 		\begin{lem}
 		The discrete optimal control problem  \eqref{eqdobjective}-\eqref{eqdstate2} admits a unique solution $(y_h,u_h)\in V_h\times U_{ad}$, and  its equivalent optimality conditions read:  find $(y_h,p_h,u_h)\in V_h\times V_h\times U_{ad}$ such that 
 		\begin{align}
 		& a(y_h,v_h)=(u_h+f,v_h), \quad \forall v_h\in V_h,\label{eqdstate3}\\
 		& a(v_h,p_h)=(y_h-y_d,v_h),\quad \forall v_h\in V_h,\label{eqdcostate2}\\
 		& (p_h+\alpha u_h,v-u_h) \geq 0,\quad \forall v\in U_{ad}.\label{eqdprojection}
 		\end{align}
 	\end{lem}
 \begin{rem}
We note that the  optimal control $u$  is not directly discretized in the objective functional \eqref{eqdobjective}, as $U_{ad}$  is infinite dimensional.  In fact, the variational inequality \eqref{eqdprojection} means that the discrete control $u_h$ is   the $L^2-$ projection of $-\frac{p_h}{\alpha}$ onto $U_{ad}$, i.e. 
 \begin{equation}\label{eqprojection3}
 u_h=P_{U_{ad}} \left(-\frac{p_h}{\alpha} \right).
 \end{equation}
This  is a key point of   the variational discretization concept. 
 In particular, if  the functions $u_0$ and $u_1$ are well-defined at any  $x\in \Omega$, then \eqref{eqprojection3} is equivalent to
 \begin{equation}\label{eqprojection4}
 u_h=min\left \{u_1,max \left \{u_0, -\frac{p_h}{\alpha} \right \} \right\}.
 \end{equation}
 \end{rem}

 	\subsection{Error estimates}
	
	Recall that $y^h\in V_h$ and $p^h\in V_h$ are the solutions to the XFEM formulations \eqref{eqdstate} and \eqref{eqdcostate}, respectively. 
	In what follows we first  show that the errors between $(y,p,u)$ and $(y_h,p_h,u_h)$, which are the solutions of the continuous  optimal control problem  \eqref{eqstate}-\eqref{eqprojection} and the discrete optimal control problem \eqref{eqdstate2}-\eqref{eqdprojection} respectively,  are bounded from above by the errors  between $(y,p)$ and $(y^h,p^h)$.
 	
 	\begin{thm}\label{th1}
 		Let $(y,p,u)\in H_0^1(\Omega)\times H_0^1(\Omega)\times U_{ad}$ and $(y_h,p_h,u_h)\in V_h\times V_h\times U_{ad}$ be the solutions to the continuous  problem  \eqref{eqstate}-\eqref{eqprojection} and the discrete  problem \eqref{eqdstate2}-\eqref{eqdprojection}, respectively. 
		Then we have 
 		\begin{eqnarray}
 		\alpha^{\frac12} \|u-u_h\|_0+\|y-y_h\|_0 &\lesssim& \|y-y^h\|_0+{\alpha^{-\frac12}}\|p-p^h\|_0, \label{t1} \\
 		\|p-p_h\|_{0} &\lesssim& \|p-p^h\|_{0} + \|y-y_h\|_{0}, \label{t2} \\
 		|y-y_h|_1 &\lesssim& |y-y^h|_1+\|u-u_h\|_{0},  \label{t3} \\
 		|p-p_h|_1&\lesssim& |p-p^h|_1+\|y-y_h\|_{0}.\label{t4}
 		\end{eqnarray}
 	\end{thm}
 	\begin{proof}
 		We first show \eqref{t1}. From \eqref{eqdstate}-\eqref{eqdcostate} and  \eqref{eqdstate2}-\eqref{eqdcostate2} it follows
 		\begin{align}
 		a(y_h-y^h,v_h)&=(u_h-u,v_h), \quad \forall v_h \in V_h, \label{eqt1} \\
 		a(v_h,p_h-p^h)&=(y_h-y,v_h), \quad \forall v_h \in V_h, \label{eqt2}
 		\end{align}
 		which yields 
 		\begin{align}
 		(y_h-y,y_h-y^h)&=a(y_h-y^h,p_h-p^h)=(u_h-u,p_h-p^h). \label{eqt3} 
 		\end{align}
 		By \eqref{eqprojection} and   \eqref{eqdprojection} we get 
 		\begin{align*}
 		(\alpha u+p,u_h-u) \geq 0, \quad 
 		(\alpha u_h+p_h,u-u_h) \geq 0,
 		\end{align*}
 		which imply   $$(\alpha (u-u_h)+p-p_h,u_h-u) \geq 0.$$
		This inequality, together with \eqref{eqt3}, indicates
 		\begin{align*}
 		\alpha \|u-u_h\|^2_0 & \leq (u_h-u,p-p_h) \\
 		&=(u_h-u,p-p^h)+(u_h-u,p^h-p_h) 	\\
 		&=(u_h-u,p-p^h)+(y_h-y,y^h-y_h)  \\
 		&\leq \frac{1}{2}\left(\alpha\|u-u_h\|^2_0+\frac{1}{\alpha}\|p-p^h\|^2_0\right)-(y-y_h,y-y_h) +(y_h-y,y^h-y)\\	
 		&\leq \frac{1}{2}\left(\alpha \|u-u_h\|^2_0+\frac{1}{\alpha}\|p-p^h\|^2_0\right) -\frac{1}{2}\|y-y_h\|_0^2+\frac{1}{2}\|y-y^h\|_0^2,
 		\end{align*}
 		which implies \eqref{t1}. 
 		
 		Secondly, let us prove \eqref{t2}.  Since $p_h-p^h\in V_h\subset H_0^1(\Omega)$ (cf. \eqref{subset}),  by \eqref{eqt2} we have
 		\begin{align*}
 		|| p_h-p^h||^2_{0}&\lesssim |p_h-p^h|_1^2\\
 		&= a(p_h-p^h,p_h-p^h)=(y_h-y, p_h-p^h)\\
 		&\lesssim ||y_h-y||_{0}|| p_h-p^h||_{0},
 		\end{align*}
 		which,  together with the triangle inequality,  leads to
 		\begin{align*}
 		\|p-p_h\|_0 & \leq \|p-p^h\|_0+\|p^h-p_h\|_0 \\
 		& \lesssim \|p-p^h\|_0+\|y_h-y\|_0,
 		\end{align*}
 		i.e.  \eqref{t2} holds.
 		
 		Thirdly, let us derive \eqref{t3}.  In view of  \eqref{eqt1}, we obtain
 		\begin{align*}
 		|y_h-y^h|_1^2 &= a(y_h-y^h,y_h-y^h)= (u_h-u,y_h-y^h) \\
 		&\leq \|u-u_h\|_{0} \|y_h-y^h\|_0 \\
 		&\lesssim \|u-u_h\|_{0}|y_h-y^h|_1,
 		\end{align*}
 		which, together with the triangle inequality, indicates \eqref{t3}.
 		
 		Finally, let us show \eqref{t4}.From  \eqref{eqt2} we get
 		\begin{align*}
 		|p_h-p^h|_1^2&= a(p_h-p^h,p_h-p^h)=(y_h-y, p_h-p^h)\\
 		&\lesssim||y_h-y||_{0}|| p_h-p^h||_{0}\\
 		&\lesssim ||y_h-y||_{0}| p_h-p^h|_1,
 		\end{align*}
 	which, together with the triangle inequality, yields  \eqref{t4}.
 	\end{proof}
 	
 	Based on Theorem \ref{th1} and Lemmas \ref{lem3.1}-\ref{lem3.2}, we immediately have the following optimal error estimates.
 	\begin{thm}\label{th2}
 		Let $(y,p,u)\in H_0^1(\Omega)\times H_0^1(\Omega) \times U_{ad}$ and $(y_h,p_h,u_h)\in V_h\times V_h\times U_{ad}$ be  the solutions to the continuous problem (\ref{eqstate})-(\ref{eqprojection}) and the discrete problem (\ref{eqdstate2})-(\ref{eqdprojection}) respectively  such that  the regularity conditions \eqref {leregularity} and \eqref{regularity-p} hold. Then we have 
 		\begin{eqnarray}
 		\|u-u_h\|_{0}+\|y-y_h\|_{0} +\|p-p_h\|_{0} &\lesssim& h^2\left(\|u\|_{0}+\|f\|_0+\|y_d\|_0\right),\label{u0}\\
 		|y-y_h|_1 + |p-p_h|_1&\lesssim& h \left(\|u\|_{0}+\|f\|_0+\|y_d\|_0\right). \label{y00}
 		\end{eqnarray}
 	\end{thm}
 
 \section{Iteration algorithm}

 	 	Notice that  the optimal control problem   \eqref{eqobjective}-\eqref{eqstrongstate} without the constraint \eqref{eqcontrol1} is a linear problem, and the resultant discrete linear system is easy to solve. However,  for    the constrained    optimal control problem   \eqref{eqobjective}-\eqref{eqcontrol1}, 
			the corresponding discrete optimal control problem  \eqref{eqdobjective}-\eqref{eqdstate2} or its equivalent optimality problem \eqref{eqdstate3}-\eqref{eqdprojection} is a nonlinear system,   and we shall apply the 
		semi-smooth Newton algorithm $\cite{Hinze2009Variational}$ to solve it. 	
To describe this iteration algorithm, we first show the matrix form of the discrete system  \eqref{eqdstate3}-\eqref{eqdprojection}. 

Let $\{ \varphi_i: i=1,2,\cdots, I\}$ be  a set of basis functions of the XFE space $V_h$ with $ I=dim( V_h)$,	and $Y_h,P_h$ be column vectors consisting of corresponding degrees of freedom of $y_h,p_h$ respectively, such that 
$$y_h=( \varphi_1,  \varphi_2,\cdots, \varphi_{ I})Y_h,\quad p_h=(  \varphi_1,  \varphi_2,\cdots, \varphi_{ I})P_h.$$
Define matrices $A, M\in \Re^{ I\times  I}$ and vectors $F_1, F_2\in \Re^{ I}$ by 
 	 		  \begin{align*}
 	 		 	&A(i,j)=a(\varphi_i,\varphi_j), \quad  M(i,j)=(\varphi_i,\varphi_j),\\
				&F_1(j)=(f,\varphi_j),   \quad F_2(j)=(y_d,\varphi_j) 
 	 		 \end{align*}
  	for $ i,j =1,2,\cdots,  I$. Then  \eqref{eqdstate3} and \eqref{eqdcostate2}  are equivalent to the following matrix equations: 
  		 \begin{align} 
  		 	AY_h=MU_h+F_1,\label{matrx-equ10} \\
  		 	AP_h=MY_h-F_2. \label{matrx-equ2}
  		 \end{align}
  
  In view of \eqref{eqprojection4}, i.e. $u_h=\min \left \{u_1,\max \left \{u_0,-\frac{p_h}{\alpha} \right \} \right \}$, we   define the    in-active set $\Omega_I$ and active set $\Omega_A$ as follows:
  $$\Omega_{I}:=\left \{x\in\Omega: u_0<-\frac{p_h}{\alpha} <u_1 \right \},$$
  $$ \Omega_{A}= \Omega_{A0}\cup  \Omega_{A1}, \quad \Omega_{A0}:=\left \{x\in\Omega: -\frac{p_h }{\alpha} \leq u_0 \right \}, \ \Omega_{A1} :=\left \{x\in\Omega: -\frac{p_h}{\alpha} \geq u_1 \right \}.$$ 
  It is evident that 
\begin{equation}\label{uh-ph}
 u_h|_{\Omega_{I}}=-\frac{p_h}{\alpha},\quad u_h|_{\Omega_{A0}}=  u_0,\quad u_h|_{\Omega_{A1}}=  u_1.
 \end{equation}
Note that in the algorithm to be given  $u_0, u_1$ will be replaced by their approximations. 
 
Let $V_h^*$ be the modified finite element space of $V_h$, 
	 $$V_h^{1,*}:=\left\{
\begin{array}{ll}\small  
W_1 & \text{if}\  \frac12<\beta\leq1,\\\\
W^*_1& \text{if}\  \beta=\frac12,
\end{array}
\right. $$ 
$$V_h^{2,*}:=\left\{
\begin{array}{ll}\small 
W_2  & \text{if}\  \frac12<\beta\leq1,\\\\
W^*_2	& \text{if}\  \beta=\frac12.
\end{array}
\right. $$  
 $u_0^*,u_1^*\in V_h^*$ are $L^2-$ projections of $u_0, u_1$ onto $V_h^*$, respectively. Let  $u_h^*$ be an approximation of $u_h$ with
\begin{equation}\label{uh-ph*} 
u_h^*|_{\Omega_{I}}=u_h, \quad u_h^*|_{\Omega_{A0}}=  u_0^*,\quad u_h^*|_{\Omega_{A1}}=  u_1^*.
 \end{equation}
It is obvious that $u_h^* \in \phi V_h^*+(1-\phi) V_h^*,$  where 
  $\phi$ is  the characteristic function of $\Omega_{I}$.  
  Let $U_{h,1},U_{h,2}\in \Re^{I}$ denote the  column vectors consisting of corresponding degrees of freedom  of $\phi u_h^*$ and $(1-\phi) u_h^*$, 
  respectively, and 
 define matrices $M_1, M_2\in \Re^{ I\times I}$ by
$$M_1(i,j):=(\varphi_i,\phi  \varphi_j), \quad M_2(i,j):=(\varphi_i,(1-\phi )\varphi_j)$$
for $ i  =1,2,\cdots, I$ and $ j =1,2,\cdots,  I$.
Then 
the 
 matrix form \eqref{matrx-equ10} is modified  as 
\begin{equation}\label{modi}
AY_h= M_1 U_{h,1} +M_2  U_{h,2} +F_1.
\end{equation}

Based on \eqref{matrx-equ10}-\eqref{modi}, we can describe the semi-smooth newton algorithm  as follows.

	 		  	 \paragraph{Semi-smooth newton algorithm \\} 
			 Set $k=0,  \phi^{(0)}=0, U_{h,1}^{(0)}=0, U_{h,2}^{(0)}=0$;\\ 
 	 		
			 \noindent \textbf{Do until convergence} 
 	 		 \begin{enumerate}
 	 		 	\item Compute $y^{(k+1)}_h\in V^h$ by $$a(y_h^{(k+1)},v_h)=\left(\phi^{(k)}u_{h}^{(k)}+(1-\phi^{(k)})u_{h}^{(k)}+f,v_h\right),\forall v_h\in V^h;$$
				or, equivalently, compute $$Y_h^{(k+1)}=A^{-1}(M_1^{(k)} U_{h,1}^{(k)}+M_2^{(k)} U_{h,2}^{(k)}+F_1);$$
				
 	 		 	\item Compute $p_h^{(k+1)}\in V^h$ by $$a(v_h,p_h^{(k+1)})=(y_h^{(k+1)}-y_d,v_h),\forall v_h\in V^h;$$ or, equivalently, compute $$P_h^{(k+1)}=A^{-1}(MY^{(k+1)} -F_2)=A^{-1}\left(MA^{-1}\left(M_1^{(k)} U_{h,1}^{(k)}+M_2^{(k)} U_{h,2}^{(k)}+F_1\right) -F_2\right);$$
				
				\item Compute $$ \Omega_{I}^{(k+1)}:=\left \{x\in\Omega: u_0<-\frac{p_h^{(k+1)}}{\alpha} <u_1 \right \},$$
				$$ \Omega_{A0}^{(k+1)}:=\left \{x\in\Omega: -\frac{p_h^{(k+1)}}{\alpha} \leq u_0 \right \}, \quad\Omega_{A1}^{(k+1)}:=\left \{x\in\Omega: -\frac{p_h^{(k+1)}}{\alpha} \geq u_1 \right \},$$ 
	and the characteristic function, $\phi^{(k+1)}$, of $\Omega_{I}^{(k+1)}$;
 	 		 	\label{computeyhi}
 	 		 	\item Compute $u_{h,2}^{(k+1)}$ (or  $U_{h,2}^{(k+1)}$) with
				 $$u_{h,2}^{(k+1)}|_{ \Omega_{A0}^{(k+1)}}=u_0^*,\quad 
				 u_{h,2}^{(k+1)}|_{ \Omega_{A1}^{(k+1)}}=u_1^*,\quad u_{h,2}^{(k+1)}|_{ \Omega_{I}^{(k+1)}}=0;
				 $$
  
 \item Compute $U_{h,1}^{k+1}$ by 
  $$U_{h,1}^{k+1}=-\frac1\alpha \left(A^{-1}(MA^{-1}(M_1^{(k+1)} U_{h,1}^{(k+1)}+M_2^{(k+1)} U_{h,2}^{(k+1)}+F_1) -F_2) \right), $$
i.e.				
				$U_{h,1}^{k+1}=-\frac{1}{\alpha}(I+\frac{1}{\alpha}A^{-1}MA^{-1}M_1^{(k+1)})^{-1}( A^{-1}MA^{-1}M_2U_{h,2}^{(k+1)}+A^{-1}MA^{-1}F_1-A^{-1}F_2)$; 
 	 		 	\item $k=k+1;$
 	 		 \end{enumerate}
  		 	 		 \noindent \textbf{end}
					 
					 	It should be pointed out  that in step 3, the  active set can only  be computed approximately for XFEM, even when $u_0,u_1$ are constants, since some basis functions of  the XFE spaces  are non-linear. In actual computation we   just use their piecewise linear interpolations to replace the nonlinear basis functions so as to compute the approximate active set. We refer to  \cite{Sevilla2010Polynomial} for an efficient method to compute the active set for high order finite element methods.  

  		\section{Numerical results}

 	 		In this section, we shall provide several  numerical examples to verify the performance of the  proposed  methods, i.e.  the discrete schemes \eqref{eqdobjective}-\eqref{eqdstate2} or  \eqref{eqdstate3}-\eqref{eqdprojection} with   $V_h=V_h^1$ and $V_h=V_h^2$. We recall, cf. Remark \ref{rem3.1},  that    $V_h^1$ and $V_h^2$  are corresponding to the XFEM with a cut-off function (abbr. cut XFEM)  and   the  classic XFEM with a fixed enrichment area (abbr. classic XFEM).

 	  \begin{example}		
 		{An unconstrained problem in a crack domain.}
	\end{example}	
		 Take $\Omega=[-1,1]\times[-1,1]$ with a segment crack from the  point $(-1,0)$ to  the crack-tip $(0,0)$ (cf. Figure \ref{crack}). We choose $\alpha=0.01$, the enrichment radius $r_s=0.5$ (cf. Remark \ref{rem2.1}),  and the cut-off function $\chi(r)$ in \eqref{chi} is a polynomial with $r_0=0.01$ and $r_1=0.99$.
        Let   \begin{align*}
        y&=\sqrt{r} sin(\frac{\theta}{2})-\frac{1}{4}r^2, \\
        p&=x_2^2(1-x_2^2)(1-x_1^2)+\frac{1}{2}\sqrt{r}sin(\frac{\theta}{2})(1-x_1^2)(1-x_2^2), \\
        u&=-\frac{p}{\alpha}
        \end{align*}
        be the  analytical state, co-state and control of the    optimal control problem \eqref{eqobjective} subject to 
		\begin{equation}\label{eqstrongstaten1}
 		\left\{
 		\begin{array}{rll}
 		& -\Delta y=u+f, & \text{ in }\Omega \\
 		& y=y_b, & \text{ on }\partial\Omega.
 		\end{array}
 		\right.
 		\end{equation}
		Note that in this case $U_{ad}=L^2(\Omega)$, and $y_d$ can be obtained by $-\Delta p =y-y_d$. In particular, the discrete equation \eqref{eqdprojection} yields $ u_h=-\frac{p_h}{\alpha}$, which means $ u-u_h=-\frac{p-p_h}{\alpha}$.
		

 		We use $N\times N$ uniform triangular meshes (cf. Figure \ref{crack}). Tables \ref{xfem2example1err}-\ref{xfem1example1err} show  results of the relative errors between  $(y_h,p_h)$ and $(y,p)$ in   $H^1$ semi-norm and $L^2$ norm, and Figure \ref{xfemexample1err} shows the 	 relative errors   against the mesh size $h=2/N$. We can see that the proposed methods  yield  optimal convergence orders, i.e. first order  rates of convergence for  $|y-y_h|_1$ and  $|p-p_h|_1$, and second order rates of convergence for $|y-y_h|_0$ and  $|p-p_h|_0$. This is consistent with our theoretical results in Theorem \ref{th2}.
		
		 For comparison  we also show in Table \ref{p1example1err} and Figure \ref{p1err} the  	 relative errors of $(y_h,p_h)$ for the standard linear  element method ($P_1$ FEM) on body fitted meshes (cf. Figure \ref{crack}).  We can see that $P_1$ FEM yields only about $0.5$ order convergence rates  for  $|y-y_h|_1$ and first order convergence rates for $|y-y_h|_0$ and  $|p-p_h|_0$. This is conformable to the  theoretical results in \cite{Babu1972A}.
		 
%
		
 		

	\begin{figure}[!hbtp]
 			\centering
 			\begin{minipage}{7cm}
 				\includegraphics[width=7cm]{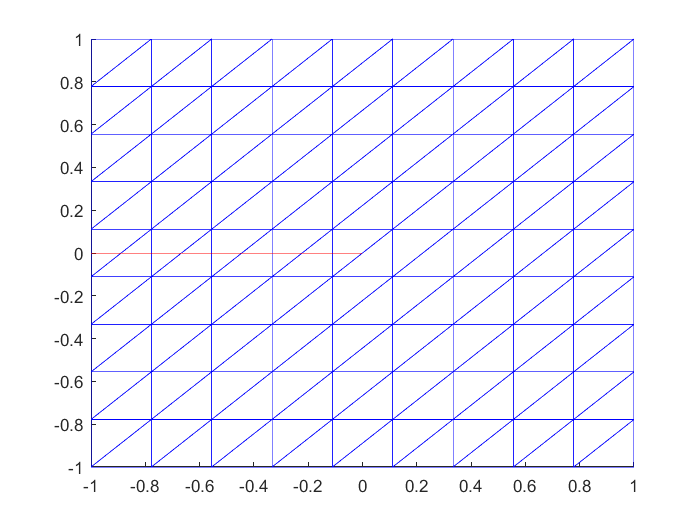}
 			\end{minipage}
 			\begin{minipage}{7cm}
 				\includegraphics[width=7cm]{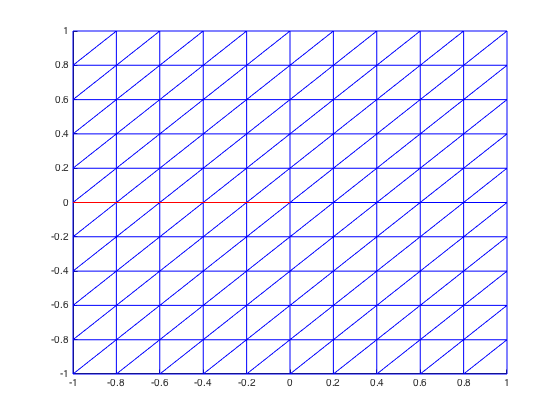}
 			\end{minipage}
 			\caption{ Domain $\Omega$ ( (the red line is the crack) and meshes  for Example 6.1: $9\times 9$ mesh for XFEMs (left) and  $10\times 10$ mesh for $P_1$ FEM.}\label{crack}
	\end{figure}		
			
	\begin{table}[!hbtp] 
 			\centering
 			\caption{Relative errors of  the cut XFEM for Example 6.1.} 		\label{xfem2example1err}	
	\begin{tabular}{|c|c|c|c|c|c|c|c|c|} 
 				\hline
 				N &$\frac{|y-y_h|_1}{|y|_1}$&order&$\frac{\|y-y_h\|_0}{\|y\|_0}$&order&$\frac{|p-p_h|_1}{|p|_1}$&order&$\frac{\|p-p_h\|_0}{\|p\|_0}$&order\\
 				\hline
 				39 & 0.1153 &    &0.0235   &    &   0.1457  &  &   0.0092  &      \\
 				\hline
 				49 & 0.0912& 1.03  & 0.0148 &   2.02 &  0.1162  &  0.99 &    0.0058  &  1.99 \\ 
 				\hline
 				59&   0.0756  &  1.01 &    0.0101   & 2.06 &    0.0966  &  0.99 &    0.0040 &   2.00 \\
 				\hline
 				69&     0.0647 &   0.99 &    0.0073 &   2.10 &    0.0827 &   0.99 &    0.0029 &   2.01 \\
 				\hline
 				79& 0.0566   & 0.99 &    0.0054 &   2.14 &    0.0723  &  0.99 &   0.0022  &  2.02 \\
 				\hline
 			\end{tabular}
 		 		\end{table}  
		
 		\begin{table}[!hbtp]

 			\centering
 			\caption{Relative errors of the classic XFEM for Example 6.1.}\label{xfem1example1err}

			\begin{tabular}{|c|c|c|c|c|c|c|c|c|}
 				\hline
 				N &$\frac{|y-y_h|_1}{|y|_1}$&order&$\frac{\|y-y_h\|_0}{\|y\|_0}$&order&$\frac{|p-p_h|_1}{|p|_1}$&order&$\frac{\|p-p_h\|_0}{\|p\|_0}$&order\\
 				\hline
 				39 & 0.0632  &  &  0.0192  &   &   0.1380  &   &  0.0084  &  \\
 				\hline
 				49 & 0.0491  & 1.11 & 0.0124   & 1.91  &  0.1103  &  0.98   & 0.0053  &  1.98 \\ 
 				\hline
 				59&   0.0403  &  1.06 &  0.0087 &   1.93  & 0.0919 &   0.99 &   0.0037   & 1.98 \\
 				\hline
 				69&    0.0344   & 1.02  & 0.0064  & 1.94  &  0.0787  &  0.99  &  0.0027 &   1.99 \\
 				\hline
 				79& 0.0288   & 1.29 &  0.0048   & 2.11  &  0.0685   & 1.02   & 0.0021 &   2.02 \\
 				\hline
 			\end{tabular} 
 		 		\end{table}
 		
 		\begin{figure}[!hbtp]
 			\centering
 			\begin{minipage}{7cm}
 				\includegraphics[width=7cm]{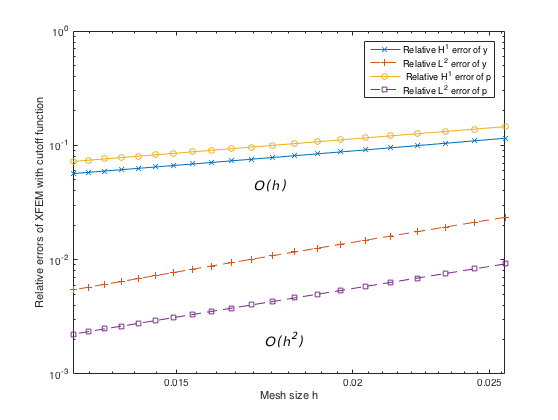}
 			\end{minipage}
 			\begin{minipage}{7cm}
 				\includegraphics[width=7cm]{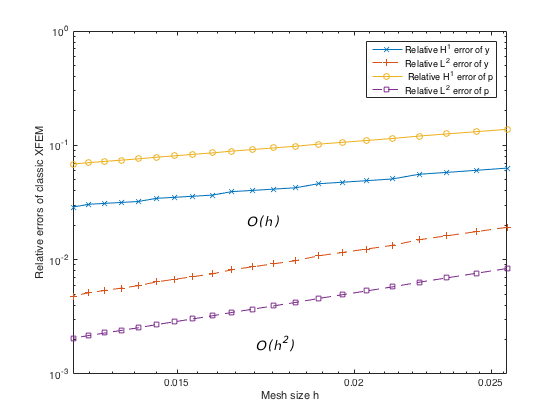}
 			\end{minipage}  \label{xfemexample1err}
 		\caption{Convergence history of the cut XFEM   (left) and  classic XFEM    (right) for Example 6.1.}
 		\end{figure}
 		

 		\begin{table}[!hbtp]
 			\centering
 			\caption{Relative errors of $P_1$ FEM for Example 6.1.} \label{p1example1err}
		\begin{tabular}{|c|c|c|c|c|c|c|c|c|}
 				\hline
 				N &$\frac{|y-y_h|_1}{|y|_1}$&order&$\frac{\|y-y_h\|_0}{\|y\|_0}$&order&$\frac{|p-p_h|_1}{|p|_1}$&order&$\frac{\|p-p_h\|_0}{\|p\|_0}$&order\\
 				\hline
 				40  &  0.2787  &    &  0.0811   &  &    0.1779  &  &   0.0206  &       \\
 				\hline
 				50 &   0.2432  &  0.61 &    0.0627  &  1.16 &    0.1492 &   0.79 &    0.0159 &   1.16
 				\\ 
 				\hline
 				60 & 0.2186   & 0.59 &   0.0511  &  1.13 &  0.1299  &  0.76  &  0.0130 &   1.11 \\
 				\hline
 				70 &  0.2002  &  0.57 &    0.0431 &   1.11 &    0.1158  &  0.74 &    0.0110 &  1.09\\
 				\hline
 				80&  0.1857 &   0.56 &    0.0372  &  1.09 &   0.1051  &  0.72 &   0.0095  &  1.07 \\
 				\hline		
 			\end{tabular}
 	 		\end{table}

 			\begin{figure}[!hbtp] 
 			\centering
 			\begin{minipage}{7cm}
 				\includegraphics[width=7cm]{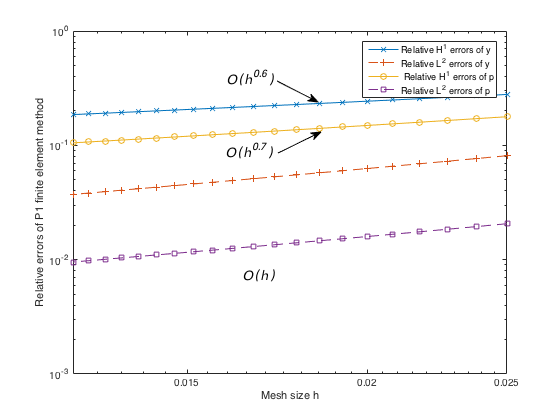}
 			\end{minipage}
 			\begin{minipage}{7cm}
 				\includegraphics[width=7cm]{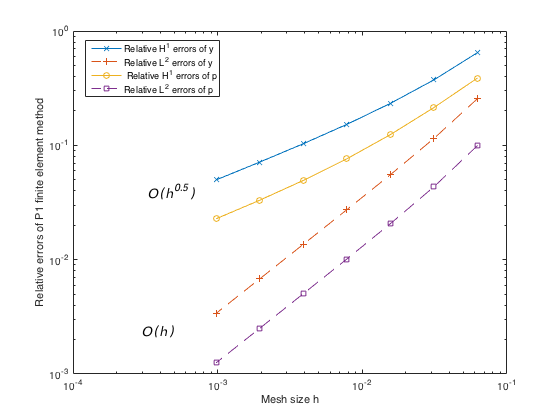}
 			\end{minipage}
 			\caption{Convergence history of $P_1$ FEM for Example 6.1.}
 			\label{p1err}
 		\end{figure}

 	\begin{example} {A constrained problem in a crack domain.} 
	
	\end{example}
	
	Let the domain $\Omega$, the  enrichment radius $r_s$, and the cut-off function $\chi(r)$ be the same as in Example 6.1. 
 		We take $\alpha=1$, $u_0=-\frac15$, $u_1=\frac15$, and  let 	
		 \begin{align*}
 		y&=\sqrt{r} sin(\frac{\theta}{2})-\frac{1}{4}r^2 \\
 		p&=x_2^2(1-x_2^2)(1-x_1^2)+\frac{1}{2}\sqrt{r}sin(\frac{\theta}{2})(1-x_1^2)(1-x_2^2) \\
 		u&=\min \left  \{\frac15,\max \left \{-\frac{p}{\alpha},-\frac15 \right \} \right \} 
 		\end{align*}
 	be the  analytical state, co-state and control of the    optimal control problem \eqref{eqobjective} subject to    \eqref{eqstrongstaten1}.  We use the same meshes as in Example 6.1. 
	
	Tables \ref{xfem2example2err}-\ref{xfem1example2err} show  results of the relative errors between  $(y_h,p_h,u_h)$ and $(y,p,u)$, and Figure \ref{figexam2} shows the 	 relative errors   against the mesh size $h=2/N$. We can see that the proposed methods  yield  optimal convergence orders, i.e. first order  rates of convergence for  $|y-y_h|_1$ and  $|p-p_h|_1$, and second order rates of convergence for $|y-y_h|_0$,  $|p-p_h|_0$, and  $|u-u_h|_0$. This is consistent with the theoretical results in Theorem \ref{th2}.


 		\begin{table}[!hbtp]
 			\centering
 				\caption{Relative errors of the cut XFEM for Example 6.2.}  \label{xfem2example2err}
	\begin{tabular}{|c|c|c|c|c|c|c|c|c|c|c|}
 				\hline
 				N &$\frac{|y-y_h|_1}{|y|_1}$&order&$\frac{\|y-y_h\|_0}{\|y\|_0}$&order&$\frac{|p-p_h|_1}{|p|_1}$&order&$\frac{\|p-p_h\|_0}{\|p\|_0}$&order&$\frac{\|u-u_h\|_0}{\|u\|_0}$&order\\
 				\hline
 				39 & 0.1089& &0.0067& &0.1456& &0.0106& &0.0108&   \\
 				\hline
 				49 & 0.0880&0.94&0.0044&1.85&0.1162&0.99&0.0067&2.00&0.0068&2.02 \\ 
 				\hline
 				59& 0.0738&0.94&0.0031&1.89&0.0966&0.99&0.0046&2.02&0.0047&1.97 \\
 				\hline
 				69&0.0636&0.95&0.0023&1.89&0.0827&0.99&0.0034&2.04& 0.0034 &2.05 \\
 				\hline
 				79&0.0559&0.96&0.0018&1.92&0.0723&0.99&0.0026&2.06&0.0026 &2.00
 				\\
 				\hline
 				
 			\end{tabular}
 	 		\end{table}
		
			\begin{table}[!hbtp]
 			\centering
 				\caption{Relative errors of the classic XFEM for Example 6.2.} \label{xfem1example2err}
 		\begin{tabular}{|c|c|c|c|c|c|c|c|c|c|c|}
 				\hline
 				N &$\frac{|y-y_h|_1}{|y|_1}$&order&$\frac{\|y-y_h\|_0}{\|y\|_0}$&order&$\frac{|p-p_h|_1}{|p|_1}$&order&$\frac{\|p-p_h\|_0}{\|p\|_0}$&order&$\frac{\|u-u_h\|_0}{\|u\|_0}$&order\\
 				\hline
 				39&0.0545& &0.0024& &0.1380& &0.0093& &0.0099&   \\
 				\hline
 				49&0.0445&0.88&0.0015&1.90&0.1103&0.98&0.0059&1.97&0.0062&2.02 \\ 
 				\hline
 				59&0.0377&0.90&0.0011&1.91&0.0919&0.99&0.0041&1.97&0.0043&1.96  \\
 				\hline
 				69&0.0327&0.91&0.0008&1.91&0.0787&0.99&0.0030&1.98&0.0031&2.02 \\
 				\hline
 				79&0.0277&1.22&0.0006&2.23&0.0685&1.02&0.0023&2.04&0.0024&1.99\\
 				\hline
 				
 			\end{tabular}
 		\end{table}

			\begin{figure}[!hbtp]
 			\centering
 			\begin{minipage}{7cm}
 				\includegraphics[width=7cm]{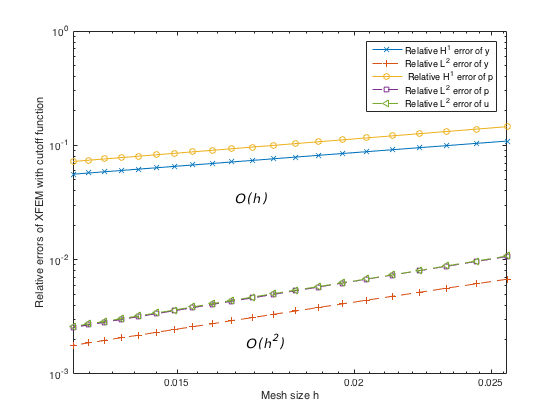}
 			\end{minipage}
 			\begin{minipage}{7cm}
 				\includegraphics[width=7cm]{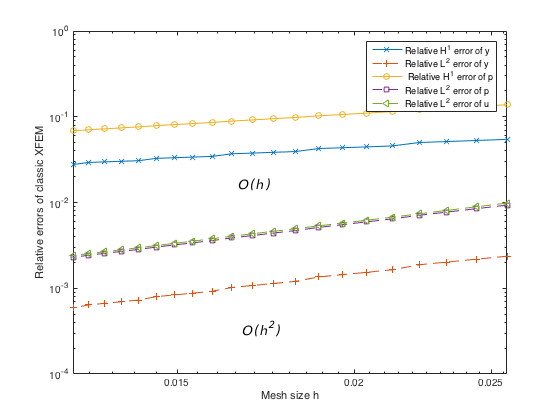}
 			\end{minipage}
 		 	\caption{Convergence history of the cut XFEM   (left) and  classic XFEM    (right) for Example 6.2.}
			\label{figexam2}
 		\end{figure} 

\begin{example} { A constrained problem in a non-convex domain.} 
\end{example} 
Let  $\Omega$ be a  $\frac34$ unit circle (Figure \ref{lshape}),  and take $\alpha=0.01$.  
We consider  the    optimal control problem \eqref{eqobjective} subject to 
 		\begin{equation}\label{63}
 		\left\{
 		\begin{array}{rll}
 		& -\Delta y+y=u+f, & \text{ in }\Omega, \\
 		& y=0, & \text{ on }\partial\Omega
 		\end{array}
 		\right.
 		\end{equation}
	with the control constraint
\begin{equation}\label{eqcontrol2}
-0.3\leq u \leq 1, \text{ a.e. on } \Omega.
\end{equation}
  Set $r_s=0.5$,  and let $\chi(r)$ be a polynomial with $r_0=0.01$, $r_1=0.99$, and   let 
 \begin{align*}
y&=(r^{\frac32}-r^{\frac52}) sin(\lambda \theta ),\\
p&=\alpha (r^{\frac32}-r^{\frac52}) sin(\lambda \theta ), \\
u&=\min \left  \{1,\max \left \{-\frac{p}{\alpha},-0.3 \right \} \right \}
\end{align*}
 be the  analytical state, co-state and control, respectively. We  note 	that  the co-state $p$ satisfies 
 	\begin{equation*}
 		\left\{
 		\begin{array}{rll}
 		& -\Delta p+p =y-y_d & \text{ in }\Omega, \\
 		& p=0, & \text{ on }\partial\Omega.
 		\end{array}
 		\right.
 		\end{equation*}
	
	We apply the  MATLAB mesh generator  $distmesh2d$ (\cite{Persson2004A})   to generate  quasi-uniform triangular meshes (Figure \ref{lshape}): for $h=1/4, 1/8, 1/12, 1/16, \cdots$, 
\begin{align*}\footnotesize
&fd=@(q) (sqrt(sum(q.^2,2))-1)+1*(q(:,1)> 0+eps\  \& \ q(:,2)< 0 -eps );\\
&[q,t]=distmesh2d(fd,@huniform,h,[-1,-1;1,1;],[(0:h:1)',(0:h:1)'*0;(0:h:1)'*0,(-1:h:0)';]);
\end{align*}
Tables \ref{tb:7}-\ref{tb:6} show  results of the relative errors between  $(y_h,p_h,u_h)$ and $(y,p,u)$, and Figure \ref{figexam2} shows the 	 relative errors   against  the number of the mesh nodes, $ND$. It is known that the optimal convergence orders of the errors (against $ND$) in $H^1$ semi-norm and $L^2$ norm are $1/2, 1$ respectively. We can see   that the proposed methods  yield  optimal convergence rates. We note that in $\cite{Apel2007Optimal}$,   graded meshes and a post-processing procedure were used to acquire optimal convergence for the  $P_1$ element. 

In Figures \ref{fg:8}-\ref{fg:10}, we also show the cut XFEM solutions of the state, control and boundary of the active set at the mesh with $h=1/8$.

\begin{figure} [!hbtp]
	\centering
			\begin{minipage}{3.5cm}
 				\includegraphics[width=3.5cm]{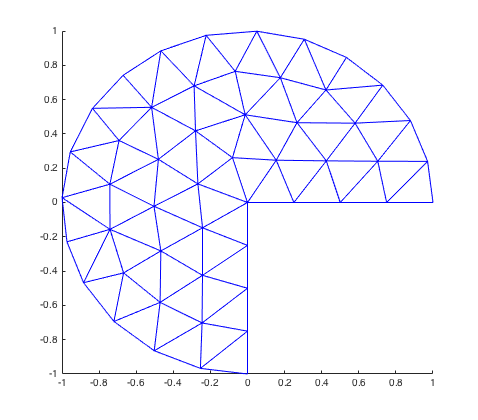}
 			\end{minipage}
 			\begin{minipage}{3.5cm}
 				\includegraphics[width=3.5cm]{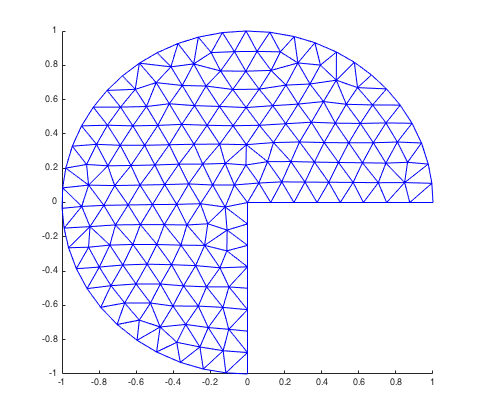}
 			\end{minipage}
			\begin{minipage}{3.5cm}
 				\includegraphics[width=3.5cm]{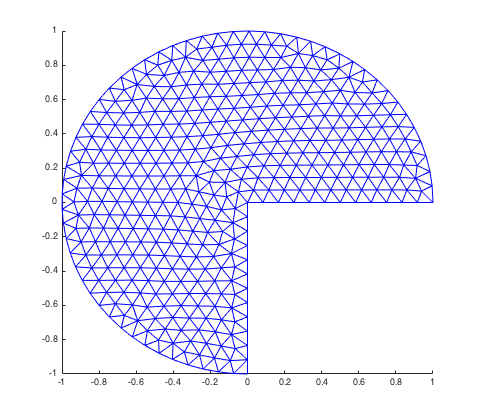}
 			\end{minipage}
                       \begin{minipage}{3.5cm}
 				\includegraphics[width=3.5cm]{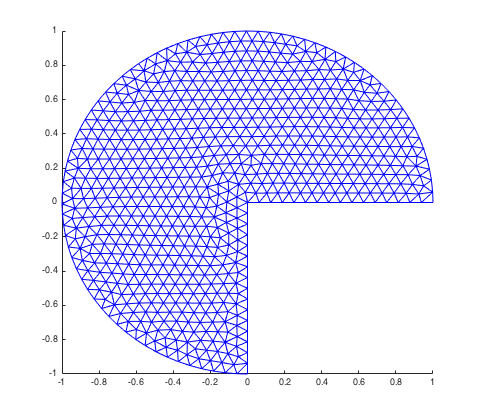}%
 			\end{minipage}
	\caption{Domain and meshes for Example 3: $h=1/4 (ND=51), 1/8 (ND=186),1/12 (ND=419), 1/16 (ND=729)$ (left to right)}
	\label{lshape}
	
\end{figure}

 		\begin{table}[!hbtp]
	\centering
	\caption{Relative errors of the cut XFEM  for Example 6.3.} \label{tb:7}
	\begin{tabular}{|c|c|c|c|c|c|c|c|c|c|c|}
		\hline
		$1/h$ &$\frac{|y-y_h|_1}{|y|_1}$&order&$\frac{\|y-y_h\|_0}{\|y\|_0}$&order&$\frac{|p-p_h|_1}{|p|_1}$&order&$\frac{\|p-p_h\|_0}{\|p\|_0}$&order&$\frac{\|u-u_h\|_0}{\|u\|_0}$&order\\
		\hline

4&   0.1535 &   &    0.0560  &    &  0.2421  &   &    0.2801 &    &   0.2230 &   \\
\hline
8 &    0.0776  &  0.53 &    0.0143 &    1.06 &    0.0934 &   0.74  &  0.0778 &    0.99 &  0.0553   & 1.08 \\
\hline
12 & 0.0506  &  0.53 &    0.0062 &    1.02 &    0.0555&   0.64 &    0.0343&    1.01 &   0.0239  &   1.03 \\
\hline
16 &   0.0383 &    0.50 &    0.0036 &    1.01 &   0.0403 &   0.58 &    0.0190  &   
1.06 &   0.0133 &   1.05 \\
		\hline
		
	\end{tabular}
\end{table}

 		\begin{table}[!hbtp]
	\centering
	\caption{Relative errors of the classic XFEM for Example 6.3.} \label{tb:6}
	\begin{tabular}{|c|c|c|c|c|c|c|c|c|c|c|}
		\hline
		$1/h$ &$\frac{|y-y_h|_1}{|y|_1}$&order&$\frac{\|y-y_h\|_0}{\|y\|_0}$&order&$\frac{|p-p_h|_1}{|p|_1}$&order&$\frac{\|p-p_h\|_0}{\|p\|_0}$&order&$\frac{\|u-u_h\|_0}{\|u\|_0}$&order\\
		\hline
		
   4 &  0.0965 &   &    0.0388  &   &   0.1628 &    &    0.1758 &   &    0.1368 &  \\
   \hline
8 &    0.0567 &    0.41 &    0.0100 &    1.05 &    0.0657 &   0.70 &   0.0471 & 1.02 &    0.0359 &   1.03 \\
\hline
12 &    0.0378 &   0.50 &   0.0044 &   1.01 &    0.0410 &   0.58 &    0.0228   & 0.89   & 0.0165 &   0.96 \\
\hline
16&   0.0287 &   0.49 &    0.0025 &   1.03 &    0.0302 &    0.55 &    0.0133    & 0.98 &   0.0094 &   1.01 \\
		\hline
		
	\end{tabular}
\end{table}

 		\begin{figure}[!hbtp]
	\centering
	\begin{minipage}{7cm}
		\includegraphics[width=7cm]{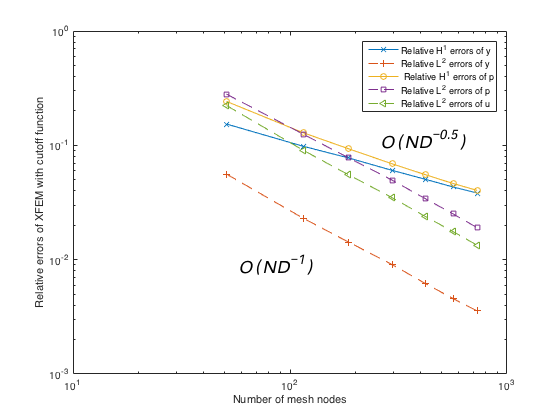}
		\end{minipage}
	\begin{minipage}{7cm}
		\includegraphics[width=7cm]{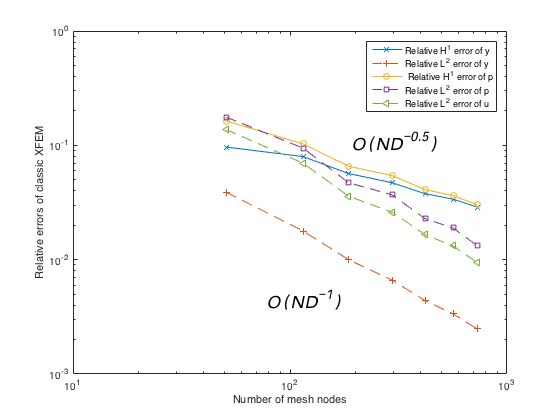}

	\end{minipage}
		\caption{Convergence history of the cut XFEM (left) and  classic XFEM (right) for Example 6.3.}
\end{figure} 


\begin{figure}[!hbtp]
	\centering
	\begin{minipage}{7cm}
		\includegraphics[width=7cm]{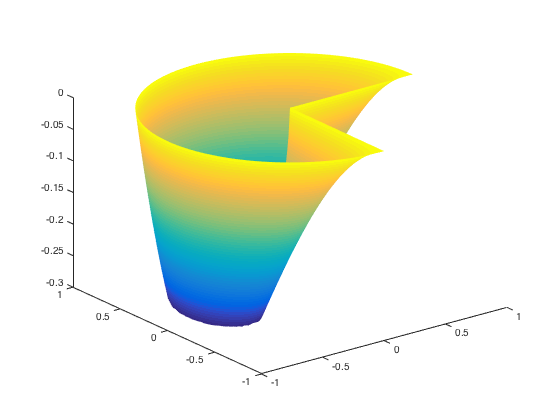}
	\end{minipage}
	\begin{minipage}{7cm}
		\includegraphics[width=7cm]{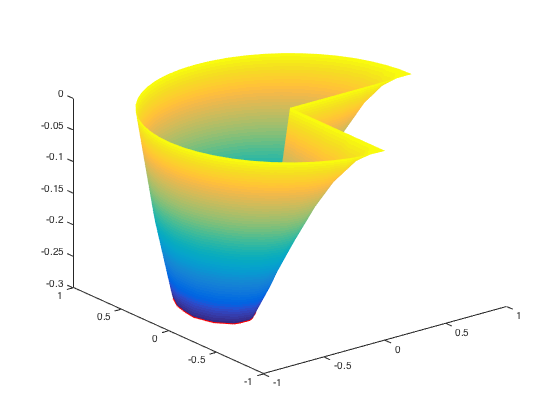}
	\end{minipage}
\caption{The exact  control $u$ (left)  and the discrete control $u_h$ by the cut XFEM (right)   for Example 6.3.} \label{fg:8}
\end{figure} 

\begin{figure}[!hbtp]
	\centering
	\begin{minipage}{7cm}
		\includegraphics[width=7cm]{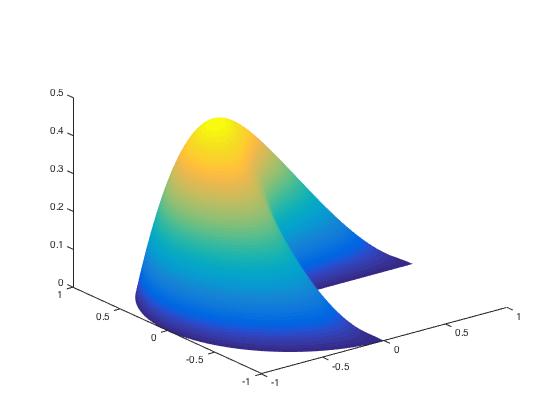}
	\end{minipage}
	\begin{minipage}{7cm}
		\includegraphics[width=7cm]{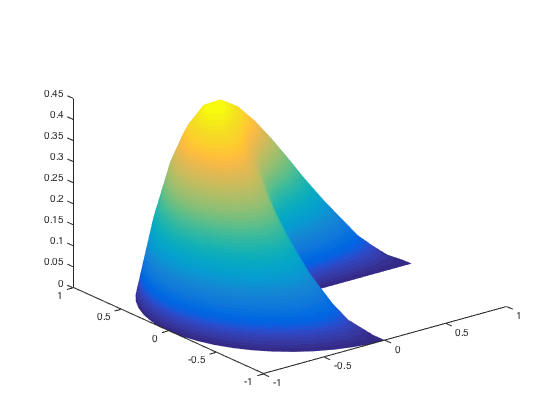}
	\end{minipage}
	\caption{The exact state $y$ (left) and  the discrete state $y_h$ by the cut XFEM  for Example 6.3.} \label{fg:9}
\end{figure} 

\begin{figure}[!hbtp]
	\centering
	\begin{minipage}{7cm}
		\includegraphics[width=7cm]{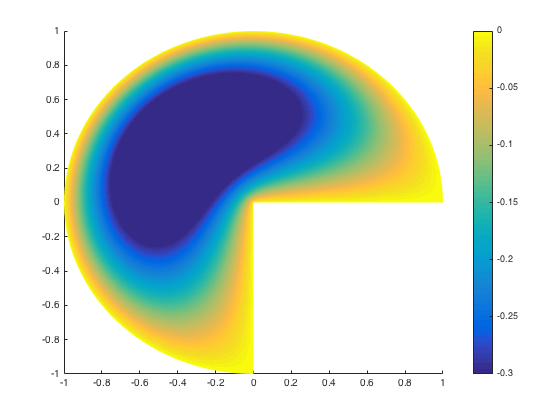}
	\end{minipage}
	\begin{minipage}{7cm}
		\includegraphics[width=7cm]{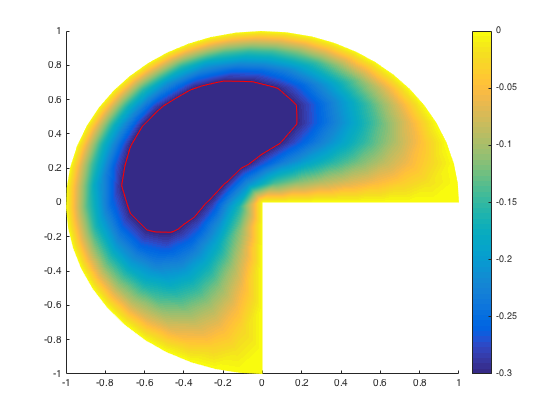}
	\end{minipage}
	\caption{The  active and inactive sets of the exact control $u$  (left) and the sets of the discrete control $u_h$ by the cut XFEM with $h=1/8$:    the red line is the boundary of active set for Example 6.3. } \label{fg:10}
\end{figure}



\end{document}